\documentclass[a4paper,11pt]{article}
\usepackage[left=20mm,right=20mm,textheight=700pt]{geometry} 

\usepackage{type1cm}
\usepackage{amsmath, amsthm, amscd, amsfonts, amssymb, graphicx, color}
\usepackage[abbrev]{amsrefs} 
\usepackage{hyphenat}
\usepackage{hyperref}
\usepackage[english]{babel}

\usepackage{enumitem}
\usepackage{cleveref} 
\usepackage{subcaption}
\usepackage{accents} 
\makeatletter
\DeclareRobustCommand{\cev}[1]{%
  \mathpalette\do@cev{#1}%
}
\newcommand{\do@cev}[2]{%
  \fix@cev{#1}{+}%
  \reflectbox{$\m@th#1\vec{\reflectbox{$\fix@cev{#1}{-}\m@th#1#2\fix@cev{#1}{+}$}}$}%
  \fix@cev{#1}{-}%
}
\newcommand{\fix@cev}[2]{%
  \ifx#1\displaystyle
    \mkern#23mu
  \else
    \ifx#1\textstyle
      \mkern#23mu
    \else
      \ifx#1\scriptstyle
        \mkern#22mu
      \else
        \mkern#22mu
      \fi
    \fi
  \fi
}
\makeatother

\newcommand{\sign}{\mathop{\mathrm{sgn}}}

\renewcommand{\Re}{\mathop{\mathrm{Re}}}

\newcommand{\kla}{\scalebox{0.7}{$\sim$}}
\newcommand{\ger}{{\scalebox{0.7}{$\rightarrow$}}}
\newcommand{\gel}{{\scalebox{0.7}{$\leftarrow$}}}
\newcommand{\kzg}{{\scalebox{0.7}{$\leadsto$}}}
\newcommand{\kindone}{{1$^{\text{st}}$ }}
\newcommand{\kindtwo}{{2$^{\text{nd}}$ }}

\allowdisplaybreaks[1] 

\newtheorem{theorem}{Theorem}
\newtheorem{proposition}[theorem]{Proposition}%
\newtheorem{corollary}[theorem]{Corollary}
\newtheorem{lemma}[theorem]{Lemma}

\newtheorem{definition}[theorem]{Definition}%
\newtheorem{example}[theorem]{Example}%

\newenvironment{myproof}{\begin{proof}}{\end{proof}}

\begin{document}

\title{Hermitian Adjacency Matrices of Mixed Graphs}

\author{Mohammad Abudayah\footnote{School of Basic Sciences and Humanities, German Jordanian University, Amman, Jordan, mohammad.abudayah@gju.edu.jo}, 
Omar Alomari\footnote{School of Basic Sciences and Humanities, German Jordanian University, Amman, Jordan, omar.alomari@gju.edu.jo},
Torsten Sander\footnote{Fakult\"at f\"ur Informatik, Ostfalia Hochschule f\"ur angewandte Wissenschaften, Wolfenb\"uttel, Germany, t.sander@ostfalia.de}}

\maketitle

\begin{abstract}
The traditional adjacency matrix of a mixed graph is not symmetric in general, hence its eigenvalues may be not real.
To overcome this obstacle, several authors have recently defined and studied various Hermitian adjacency matrices of digraphs or mixed graphs.
In this work we unify previous work and offer a new perspective on the subject by introducing the concept of monographs. 
Moreover, we consider questions of cospectrality.

\medskip



\end{abstract}

\section{Introduction}\label{intro}

Algebraic graph theory strives to relate the structural properties of graphs to the
algebraic properties of objects associated with them. Specifically, in spectral graph theory
the eigenvalues and eigenvectors of matrices associated with graphs are studied. Most traditionally,
the object of interest would be the adjacency matrix of some undirected graph, i.e., the square matrix $[a_{uv}]$
such that $a_{uv}=1$ if there is an edge between vertices $u$ and $v$, otherwise $a_{uv}=0$.
By construction, the adjacency matrix of an undirected graph is symmetric. Hence theorems from
linear algebra dealing with non-negative symmetric matrices can be readily applied to obtain a number of
desirable spectral properties. Most notably, the spectrum of the adjacency matrix is real.
Moreover, there exists a basis of pairwise orthogonal eigenvectors.

However, when dealing with directed or mixed graphs the definition of the adjacency matrix needs to be changed to accommodate the fact
that the adjacency relation of vertices is no longer symmetric. For a digraph we set $a_{uv}=1$ if there is an arc from vertex number
$u$ to $v$ and $a_{uv}=0$ otherwise. The loss of symmetry proves a serious impediment to relating algebraic and 
structural properties to one another, cf.\ the survey \cite{Bru}.

Quite recently, the idea has been presented to modify the definition of the adjacency matrix of a directed graph, using complex numbers, in such a way
that it still properly reflects the adjacency relation but at the same time constitutes a Hermitian matrix. Let us give an overview of some efforts and results in this direction.
In \cite{BM1} the authors use the imaginary number $i$ to specify $a_{uv}=i$ and $a_{vu}=-i$ whenever there is an arc from $u$ to $v$, but not vice versa, further
$a_{uv}=1$ whenever $u$ and $v$ are mutually adjacent.
Using this definition of a Hermitian adjacency matrix, it turns out that many results from algebraic graph theory known for undirected graphs also hold 
for directed graphs or at least exist in a slightly modified or weaker version.
For example, if the underlying undirected graph of a given oriented graph is bipartite, then the spectrum of the Hermitian adjacency matrix 
is symmetric with respect to zero, but -- in contrast to the undirected case -- the reverse is not true.
Independently, the authors of \cite{Liu} introduced the same notion of a Hermitian adjacency matrix and
proved many fundamental results. Moreover, they considered the Gutman energy (which is the sum of the absolute
values of all eigenvalues) of the Hermitian adjacency matrix.
Refer to \cites{LiYu,YWGQ,jova} for other related work. 

The goal of the present paper is as follows. To begin with, we generalize and unify previous results. 
We will then introduce the concept of monographs, permitting us to view some of these results from a new perspective.
Moreover, we will analyze under which conditions a mixed graph has identical spectra for different values of $\alpha$.

\section{\bf Preliminaries}\label{sec:pd}

All graphs considered hereafter shall not contain any loops or multiple edges.
A mixed graph $D$ arises from partially orienting an undirected graph $G$, i.e.\ by turning some of 
the undirected edges into single arcs. 
Thus, between any two adjacent vertices $u,v$ of the vertex set $V(G)$ there exists either an arc from $u$ to $v$ (indicated by $u\ger{}v$), an arc from $v$ to $u$ (indicated by $u\gel{}v$), or an
undirected edge (also called a digon) between $u$ und $v$ (indicated by $u\kla{}v$). Altogether, these arcs and digons form the edge set $E(D)$ of $D$. 
The graph $G$ is called the underlying graph $\Gamma(D)$ of the mixed graph $D$.
Much of the traditional terminology (e.g.\ being regular, being connected, vertex degree $\deg(\cdot)$, maximum degree $\Delta$) that is used for undirected graphs simply carries over to mixed graphs,
in the sense that, $D$ is said to have a property whenever $\Gamma(D)$ has this property.
In particular, we say that a mixed graph contains a certain undirected subgraph (e.g.\ the path $P_k$ or the cycle $C_k$ on $k$ vertices)
if $\Gamma(D)$ contains this subgraph.
A mixed walk in $D$ is a sequence of vertices $v_1,\ldots,v_k$ of $D$ such that there is an edge between any two
subsequent vertices $v_iv_{i+1}$ in $D$.
The set of all arcs from some vertex $u$ to other vertices $v$ (resp.\ from other vertices to $u$) is denoted by $N_D^{+}(u)$ (resp.\ $N_D^{-}(u)$).
The set of all digons incident with vertex $u$ is denoted by $N_D(u)$.

The (traditional) adjacency matrix $A(G) = [a_{ij} ]$ of a given, either undirected or directed, graph $G$ on $n$ vertices 
is the real matrix of order $n\times n$ such that $a_{ij} = 1$ if there is an edge from $v_i$ to $v_j$ and $a_{ij} = 0$ otherwise. 
For directed graphs, the resulting matrix $A$ is usually non-symmetric, thus losing many desirable algebraic properties. We therefore define
the following alternative:

\begin{definition}\label{halpha}
Given a mixed graph $D$ and	a unit complex number $\alpha$, i.e.\ $\left\vert \alpha \right\vert =1$, we define the $\alpha$\hyp{}Hermitian adjacency matrix $H^\alpha(D)=[h_{uv}]$ of $D$ by
\begin{align}
h_{uv} = \begin{cases}
1 & \text{if }  u\kla v,\\
\alpha &  \text{if }  u\ger v,\\
\bar{\alpha} &  \text{if }  u\gel v,\\
0 & \text{otherwise}.
\end{cases}
\end{align}
\end{definition}

When there is no ambiguity regarding the reference graph $D$ we will often omit any symbolic reference to $D$,
e.g.\ write $H^\alpha$ instead of $H^\alpha(D)$.

Clearly, the matrix $H^\alpha$ from \Cref{halpha} is Hermitian, i.e.\ $(H^\alpha)^\ast=(H^\alpha)$ where
$M^\ast$ denotes the conjugate transpose of matrix $M$. 
By $\chi_\alpha(D,x)=\det(xI-H^\alpha(D))$, where $I$ is the identity matrix, we denote the characteristic polynomial of the matrix $H^\alpha(D)$, 
calling this the $\alpha$\hyp{}characteristic polynomial of $D$. The multiset $\sigma_\alpha(D)$ of all roots of  $\chi_\alpha(D,x)$ is called the $\alpha$\hyp{}spectrum of $D$,
as opposed to the (traditional) spectrum $\sigma(\Gamma(D))$ of the underlying undirected graph $\Gamma(D)$.
Consequently, we shall refer to the elements of  $\sigma_\alpha(D)$ as the $\alpha$\hyp{}eigenvalues of $D$. Likewise, we speak of $\alpha$\hyp{}eigenvectors. 
Note that $\alpha$\hyp{}eigenvalues are always real. 

A direct consequence of \Cref{halpha} is the following summation rule characterizing $\alpha$\hyp{}eigenvectors:
\begin{proposition}\label{srule}
	Let $D$ be a mixed graph. Then $x$ is an $\alpha$\hyp{}eigenvector of $D$ corresponding to $\alpha$\hyp{}eigenvalue $\lambda$ if and only if, for each $u\in V(D)$,
	\begin{align}\label{eq:srule}
   	\lambda x(u) = \sum_{u\kla v}{x(v)} + \left(\alpha \sum_{u\ger v}{x(v)}\right) + \left(\bar{\alpha} \sum_{u\gel v}{x(v)}\right).
	\end{align}
\end{proposition}

Throughout this paper we shall assume $\vert\alpha\vert = 1$, i.e.\ $\alpha = e^{i \theta }$ for some $\theta\in\mathbb{R}$. Moreover, we make use of the
constants $\omega:=e^{\frac{\pi}{3}i}$ (a sixth root of unity) and $\gamma:=e^{\frac{2\pi}{3}i}$ (a third root of unity).
In the context of Hermitian adjacency matrices, the former constant has been endorsed in \cite{BM2}, whereas the suitability of the latter constant will become evident later on.

\section{\bf Characteristic Polynomial of \texorpdfstring{$H^\alpha(D)$}{Halpha}}\label{sec:cp}

In this section we will expand the determinant of $H^\alpha$ and study the $\alpha$\hyp{}characteristic polynomial,
in particular with respect to the three instances $H^i$, $H^{w}$ and $H^{\gamma}$. 

A classic result from linear algebra, concerning determinant expansion, is the following:

\begin{theorem}\label{det}
	If $A=[a_{i,j}]$ is a square matrix of order $n$ then
	\begin{align}
	\det(A)=\sum_{\eta \in S_n} \sign(\eta) a_{1,\eta(1)}a_{2,\eta(2)}a_{3,\eta(3)}\dots a_{n,\eta(n)}.
	\end{align}
\end{theorem}

Decades ago, the above \namecref{det} has been applied to adjacency matrices of graphs. The permutations over which the sum ranges
can be put into correspondence with certain subgraphs of the given graph. To this end, we define the following terms and notation:

\begin{definition}
	Let $D$ be a mixed graph.
	\begin{enumerate}
		\item $D$ is called elementary if, for every component $C$ of $D$, $\Gamma(C)$ is either isomorphic to $P_2$ or $C_k$ (for some $k\geq 3$).
  	\item Let $D$ be elementary. The rank of $D$ is defined as $r(D)=n-c$, where $n=\vert V(D)\vert $ and $c$ is the number of its components.
		The co-rank of $D$ is defined as $s(D) = m-r(D)$, where $m=\vert E(D)\vert$. 
	\end{enumerate}
\end{definition}

Note that the co-rank $s(D)$ is equal to the number of $C_k$ components of $D$.

Now we are ready to state the following classic theorem by Harary (cf.\ \cite{HA}):

\begin{theorem}[Determinant expansion (Harary, 1962)]\label{origha}
Let $D$ be a graph with adjacency matrix $A(G)$. Then,
\begin{align}
\det(A(G)) = \sum_S (-1)^{r(S)} 2^{s(S)},
\end{align}
where the sum ranges over all spanning elementary subgraphs $S$ of $G$.
\end{theorem}

Following the classic proof strategy used in \Cref{origha}, the result readily generalizes to any $\alpha$\hyp{}Hermitian adjacency matrix.
But first we require the following definition:

\begin{definition}\label{halphw}
	Let $D$ be a mixed graph and $H^\alpha(D)=[h_{uv}]$. With respect to this, 
  the value $h_{\alpha}(W)$ of a mixed walk $W$ with vertices $v_1,v_2,\ldots,v_k$ is defined as
		\begin{align}
		h_{\alpha}(W) = (h_{v_1v_2}h_{v_2v_3}h_{v_3v_4}\cdots h_{v_{k-1}v_k}) \in\{\alpha^r\}_{r\in\mathbb{Z}}.
		\end{align}
\end{definition}

\begin{theorem}[Determinant expansion for $H^\alpha$]\label{ourha} Let $D$ be a mixed graph. Then
	\begin{align}
	\det(H^\alpha) = \sum_{D'} (-1)^{r(D')} \ 2^{s(D')}  \Re\left(\prod_{C} h_{\alpha}(\vec{C})\right),
	\end{align}
	where the sum ranges over all spanning elementary mixed subgraphs $D'$ of $D$,
	the product ranges over all mixed cycles $C$ in $D'$, and $\vec{C}$ is any closed walk traversing $C$.
\end{theorem}

\begin{myproof}
Consider the matrix $H^\alpha$ and apply the classic proof strategy for determinant expansion on graphs, cf.\ the proof of \Cref{origha} in \cite{Biggs}.
\end{myproof}

Considering specific values of $\alpha$, the formula in \Cref{ourha} becomes more specific, too. For  $\alpha=i$
we may rediscover a result given in \cite{Liu}.
Moreover, \Cref{ourha} immediately allows us to compute the $\alpha$\hyp{}characteristic polynomial:

\begin{corollary}\label{char}
	If $\chi_\alpha(D,\lambda) = \lambda^n+ c_1\lambda^{n-1}+c_2\lambda^{n-2}+\dots+c_n$ is the $\alpha$\hyp{}characteristic polynomial of a mixed graph $D$, then 
	\begin{align}\label{eq:ck}
	(-1)^k{c_k} = \sum (-1)^{r(D')} \ 2^{s(D')} \Re\left(\prod_{C} h_\alpha(\vec{C})\right),
	\end{align}
	where the sum ranges over all elementary mixed subgraphs $D'$ with $k$ vertices,
	the product ranges over all mixed cycles $C$ in $D'$, and $\vec{C}$ is any closed walk traversing $C$.
\end{corollary}

\begin{myproof}
	This follows immediately from the fact that $(-1)^k c_k$ equals the sum of all principal minors of $H^\alpha(D)$ with $k$ rows and columns.
\end{myproof}

\section{\bf Cospectrality}

A recurring theme in algebraic graph theory is the hunt for pairs of non-isomorphic graphs
having the same spectrum. Such graphs are called cospectral. 
In contrast to this, we shall look into the question under which conditions the
same graph has identical $\alpha$\hyp{}spectrum for different values of $\alpha$.
It comes as no surprise that such spectra may be completely different:

\begin{example}
For the mixed graph shown in \Cref{fig:ncospec} we have:
\begin{align*}
\sigma_{\gamma} & = \{2.57083, -2.3222, 1.50413, -1.19239, -0.560369\}\\ 
\sigma_{\omega} & = \{-2.93033, 2.30034, 1.15439, -0.832963, 0.308565\}\\ 
\sigma_{i}      & = \{-2.71687, 2.2803, 1.50739, -1.07082, 0.0\} 
\end{align*}
\end{example}

\begin{figure}
	\centering
		\includegraphics[scale=0.6]{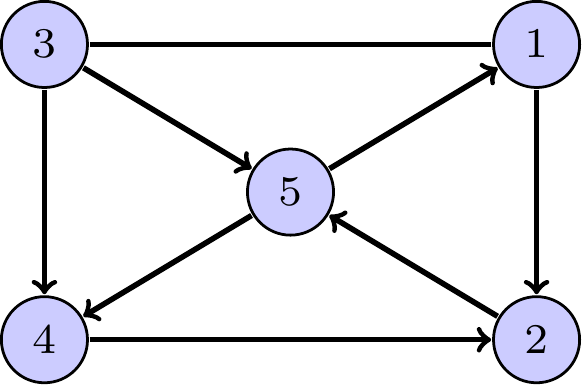}
	\caption{A mixed graph where $\sigma_{\gamma}$, $\sigma_{\omega}$, $\sigma_{i}$ are different from one another}
	\label{fig:ncospec}
\end{figure}

However, there exist mixed graphs  exhibiting the same $\alpha$\hyp{}spectrum for two different values of $\alpha$, say $\alpha_1, \alpha_2$. We call such a mixed graph
$\alpha_1$-$\alpha_2$\hyp{}cospectral. Let us give an example for a $\gamma$-$\omega$\hyp{}cospectral mixed graph:

\begin{example}\label{ex:ga_om_cospec}
The mixed graph $D$ shown in \Cref{fig:ex4} is $\gamma$-$\omega$\hyp{}cospectral, i.e.\ $\sigma_{\gamma}(D)=\sigma_{\omega}(D)$.
This is not difficult to see: With respect to \Cref{char} note that $D$ contains only one cycle. Moreover, $h_\gamma(c) \in \{\gamma, \gamma^2\}$ and 
$h_\omega(c)\in \{ \omega^2, \overline{\omega^2} \}$. Observing $\gamma=\omega^2$ we have $\chi_\gamma(D,\lambda)=\chi_\omega(D,\lambda)$.
In contrast, we remark that $\chi_\alpha(D,\lambda) \ne \chi_i(D,\lambda)$, hence $\sigma_\alpha(D)\not= \sigma_i(D)$.
\end{example}

\begin{figure*}
		\begin{center}
		\includegraphics[scale=0.6]{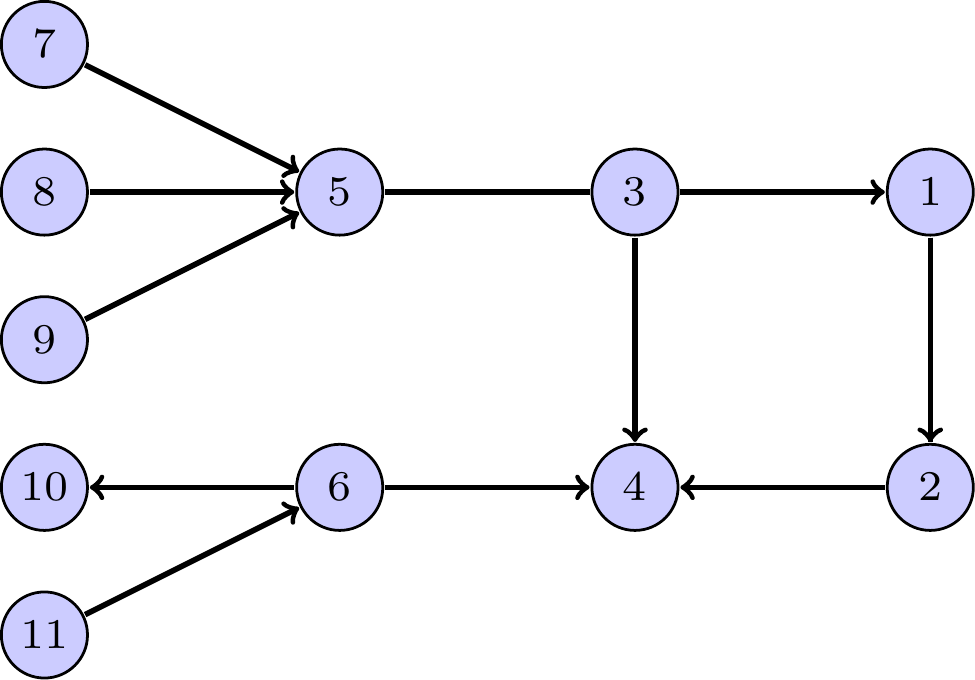}
			\caption{A $\gamma$-$\omega$\hyp{}cospectral mixed graph} 
			\label{fig:ex4}
		\end{center}
\end{figure*}

Note that for $\alpha=1$ we have $H^\alpha=A(\Gamma(D))$, so the special case of $\alpha$-$1$\hyp{}cospectrality is equivalent to asking whether the
$\alpha$\hyp{}spectrum of a mixed graph $D$ coincides with the traditional spectrum of its undirected counterpart $\Gamma(D)$.
Thus \Cref{char} immediately gives rise to the following result:

\begin{corollary}\label{tree_underly}
Let $T$ be a mixed tree. Then $\sigma_\alpha(T)=\sigma(\Gamma(T))$.
\end{corollary}

\begin{proof}
Trees do not contain cycles, hence using $\alpha=1$ in \eqref{eq:ck} instead of the given value does not change the result.
\end{proof}

Since trees are $\alpha$-$1$\hyp{}cospectral for any $\alpha$ we see that they are 
$\alpha_1$-$\alpha_2$\hyp{}cospectral for arbitrary values $\alpha_1$,$\alpha_2$.
Now, consider a mixed graph that contains cycles. Obviously, it does not matter for equation \eqref{eq:ck} if we use
$\alpha=1$ or some other specific value as long as (with respect to that other value) all factors in the involved products are equal to one.
This motivates the following definition:

\begin{definition}\label{def:mono1}
A mixed graph is an $\alpha$\hyp{}monograph (of \kindone kind) if $h_\alpha(\vec{C})=1$ for all its cycles $C$.
\end{definition}

Trivially, trees are $\alpha$\hyp{}monographs. By construction, \Cref{tree_underly} directly extends to monographs:

\begin{theorem}\label{1stspecf}
Let $D$ be an $\alpha$\hyp{}monograph (of \kindone kind). Then, $\sigma_\alpha(D)=\sigma(\Gamma(D))$.
\end{theorem}

Regarding \Cref{char} and \eqref{eq:ck}, note that
$
   h_\alpha(\vec{C})=\alpha^x\bar\alpha^y=\alpha^{x-y},
$	
where $x$ (resp.\ $y$) is the number of forward (resp.\ backward) edges encounterd while traversing $\vec{C}$.
We will tacitly make use of this fact hereafter.

\begin{corollary}\label{backforthorder}
Let $D$ be a connected mixed graph. If, for every cycle in $D$, the difference between the
numbers of encountered forward arcs and the number of backward arcs (w.r.t.\ any traversal direction)
is a multiple of the order of $\alpha$, then $D$ is an $\alpha$\hyp{}monograph.
\end{corollary}

\begin{corollary}
A connected mixed graph $G$ is an $\alpha$\hyp{}monograph for every value $\alpha$ if and only if
every cycle in $D$ contains as many forward arcs as backward arcs.
\end{corollary}

\begin{proof}
The forward implication follows directly from \Cref{backforthorder}. For the converse note that there
exist values $\alpha$ of infinite order (i.e.\ $\alpha^j = \alpha^k$ only if $j=k$).
\end{proof}

The properties of $\alpha$\hyp{}monographs deserve further investigation. 
But beforehand, we will elaborate more on cospectrality. We start with the question when the $\alpha$\hyp{}spectra for the 
special values $\gamma$ and $\omega$ coincide:

\begin{theorem}\label{gaomcospeceven}
Let $D$ be a mixed graph. If, for any cycle in $D$, the difference between the
numbers of encountered forward arcs and the number of backward arcs (w.r.t.\ any traversal direction)
is even, then $D$ is $\gamma$-$\omega$\hyp{}cospectral.
\end{theorem}

\begin{proof}
Under the given assumptions there will be only even powers of $\alpha$ in \eqref{eq:ck}.
But $\Re(\gamma^{2k})=\Re(\omega^{2k})$.
\end{proof}

\begin{corollary}\label{mixbipspec}
Let $D$ be a mixed bipartite graph. Suppose that every cycle in $D$ contains an even number of digons. Then $D$ is $\gamma$-$\omega$\hyp{}cospectral.
\end{corollary}

\begin{proof}
A bipartite graph contains only even cycles. Consider such a cycle $C$.
Subtracting an even number of digons, we conclude that $C$ contains an even number of arcs. Traversing $C$, these arcs consist 
of an even number of forward arcs and an odd number of backward arcs -- or vice versa. Hence \Cref{gaomcospeceven} can be applied.
\end{proof}

\begin{corollary}\label{bipspec}
Every oriented bipartite graph is $\gamma$-$\omega$\hyp{}cospectral.
\end{corollary}


As already mentioned in \Cref{ex:ga_om_cospec}, the graph in \Cref{fig:ex4} is $\gamma$-$\omega$\hyp{}cospectral.
This follows from \Cref{bipspec} since it is an oriented bipartite graph.

Using the ideas from the proof of \Cref{mixbipspec}, one can generalize as follows:

\begin{theorem}\label{gaomcospecevenodd}
Let $D$ be a mixed graph. 
Suppose that every even cycle in $D$ contains an even number of digons and every odd cycle in $D$ contains
an odd number of digons. Then $D$ is $\gamma$-$\omega$\hyp{}cospectral.
\end{theorem}

We have shown how \Cref{char} can be used as a tool investigating cospectrality.
In view of $\alpha_1$-$\alpha_2$\hyp{}cospectrality it is sufficient to require
that, for every elementary mixed subgraph $D'$, the real part of the associated product
in \eqref{eq:ck} is the same for both values $\alpha_1$ and $\alpha_2$.
As a slightly coarser condition one could require that $h_{\alpha_1}(\vec{C})=h_{\alpha_2}(\vec{C})$
for all cycles $C$ of $D'$. In view of this, requiring uniform cospectrality for all
values of $\alpha_1$,$\alpha_2$ (hence including value $1$) amounts to 
the condition $h_\alpha(\vec{C})=1$ stated in \Cref{def:mono1}.
In view of this, one can devise modifications which, given some $\alpha$\hyp{}monograph $D$, can be used to construct 
arbitrarily many derived $\alpha$\hyp{}monographs containing $D$ as a subgraph:

\begin{theorem}\label{extconstr}
Let $D$ be an $\alpha$\hyp{}monograph. Let $U$ be a connected undirected subgraph of $D$. Fix a set $M$ of new vertices
and subsets $V_x\subset V(U)$, for $x\in M$. Connect each vertex $x\in M$ to $V_x$ such that
either $N^+_{D}(x) = V_x$ or $N^-_{D}(x) = V_x$. Then the resulting mixed graph $\tilde D$ is  an $\alpha$\hyp{}monograph.
\end{theorem}

\begin{proof}
The newly added vertices and their adjacent edges may introduce new cycles. 
Let $C$ be such a cycle in $\tilde D$  and $\vec{C}$ any traversal of $C$.
By construction, the predecessor $v_x$ of $x$ and its successor $w_x$ (w.r.t.\ $\vec{C}$) belong to $U$. 
Consider the residual graph $R$ obtained by removing all edges of $U$ from $D$.
Every path $W$ in $R$ joining two vertices $u_1, u_2$ that originally belong to the subgraph $U$ in $D$ satisfies $h_\alpha(W)=1$. 
To see this, add any path $W'$ between $u_1$ and $u_2$ in $U$ to obtain a cycle $C_W$ in $D$. 
Choosing matching traversals for $W$ and $W'$, we get $h_\alpha(\vec{C_W})=h_\alpha(\vec{W})h_\alpha(\vec{W'})$.
Observe $h_\alpha(\vec{W'})=1$ since $U$ is an undirected subgraph of $D$.
$D$ is an $\alpha$\hyp{}monograph, so $h_\alpha(\vec{C_W})=1$ and therefore $h_\alpha(\vec{W})=1$.
$C$ can be segmented into paths of three possible types as follows: Paths within $U$, paths within $R$ and
the segment from $v_x$ via $x$ to $w_x$. The latter segment can only be $v_x\ger x \gel w_x$ or $v_x\gel x \ger w_x$.
Both contribute a factor of $1$ to the product $h_\alpha(\vec{C})$, as do the two segment types mentioned first. Overall, $h_\alpha(\vec{C})=1$.
\end{proof}

\begin{example}
\Cref{fig:exmon} illustrates the construction mentioned in \Cref{extconstr}. The vertices no.\ $12$ und $13$ have been newly added.
\end{example}

\begin{figure}
		\begin{center}
		\includegraphics[scale=0.6]{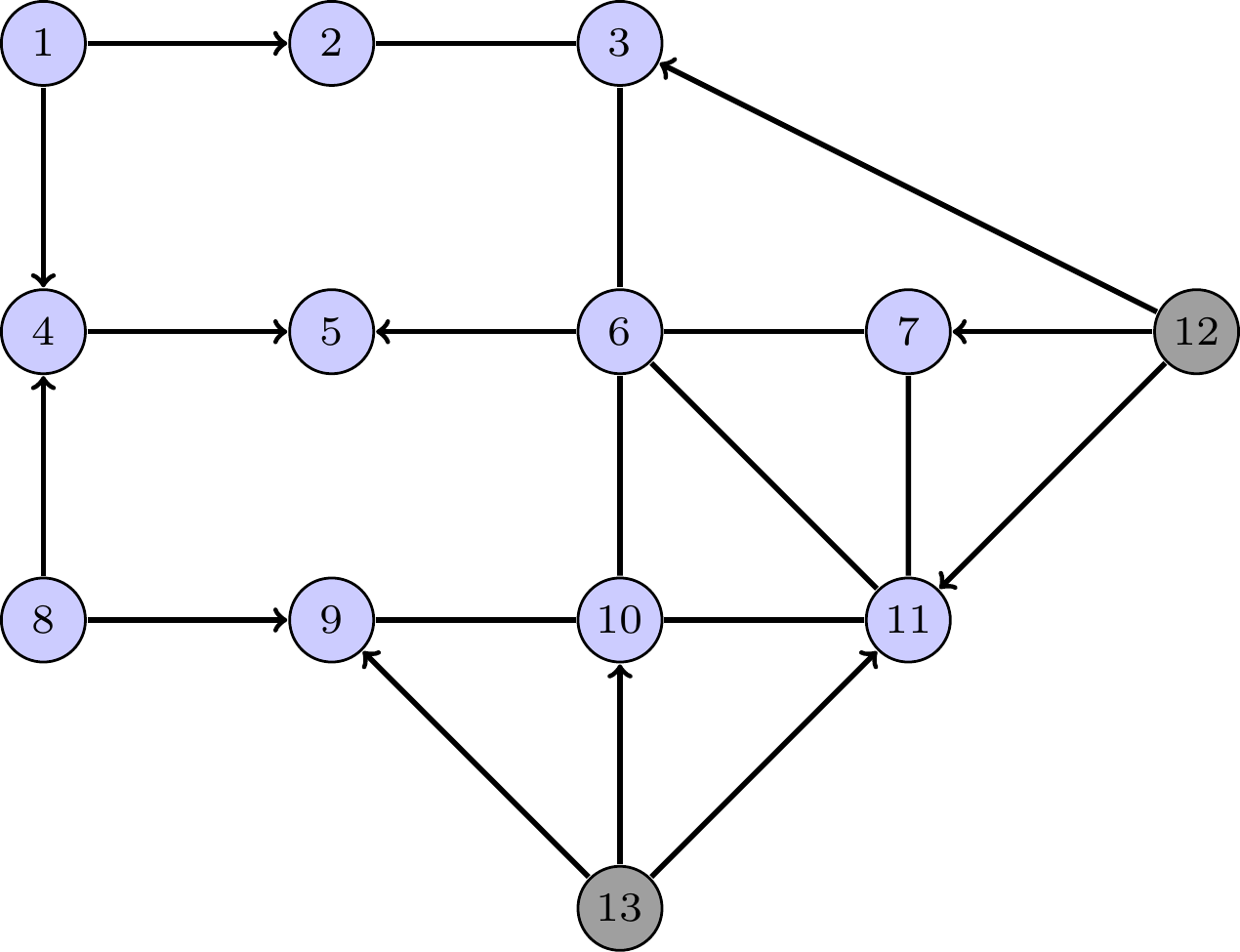}
			\caption{Extending an $\alpha$\hyp{}monograph} 
			\label{fig:exmon}
		\end{center}
\end{figure}

\section{\bf Monograph Structure}\label{sec:monstruc}

\Cref{def:mono1} characterizes $\alpha$\hyp{}monographs by a condition concerning
the traversal of cycles. In the following, we will render the implications of this condition more
tangible, by way of studing mixed walks.
To this end, let $D$ be a connected mixed graph. Fix any $u\in V(D)$ and consider some mixed walk $W$ in $D$, say $u=r_1,\ldots,r_k$.
Denote the contained partial walks $r_1,\ldots,r_j$ by $W_j$ (for $j=1,\ldots,k$). Consequently,
\begin{align}\label{eq:stfunc1}
h_{\alpha}({W_1}) & = 1 
\end{align}
and
\begin{align}\label{eq:stfunc2}
h_{\alpha}({W_{j+1}}) & = 
   \begin{cases} 
	    h_{\alpha}({W_j}) & \text{ if } r_j \kla r_{j+1} \\
	    \alpha h_{\alpha}({W_j}) & \text{ if } r_j\ger r_{j+1} \\
	    \bar\alpha h_{\alpha}({W_j}) & \text{ if } r_j\gel r_{j+1}  \\
	 \end{cases}
\end{align}
for $j=1,\ldots,k-1$.

\begin{example}\label{ex:wj}
\Cref{fig:walk} shows the values $h_{\gamma}({W_j})$ -- each value written near the respective $j$-th vertex along the walk -- 
for three different mixed walks $W$ in a mixed example graph.
\end{example}

\begin{figure}
	\centering
	\begin{subfigure}{.45\textwidth}
		\includegraphics[scale=0.6]{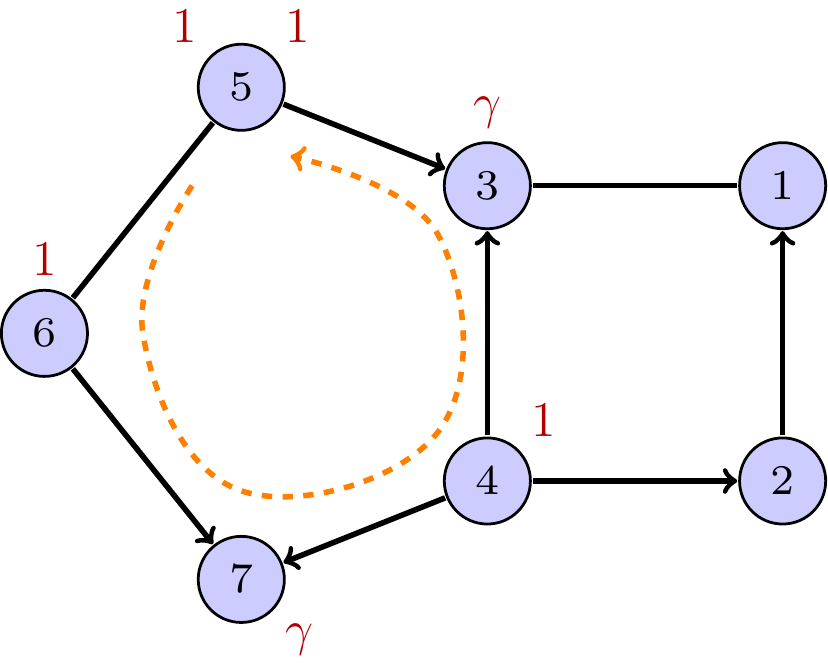}
\caption{Walk from/to vertex no.\ 5}
  \end{subfigure}
	\hspace*{0.75cm}
	\begin{subfigure}{.45\textwidth}
		\includegraphics[scale=0.6]{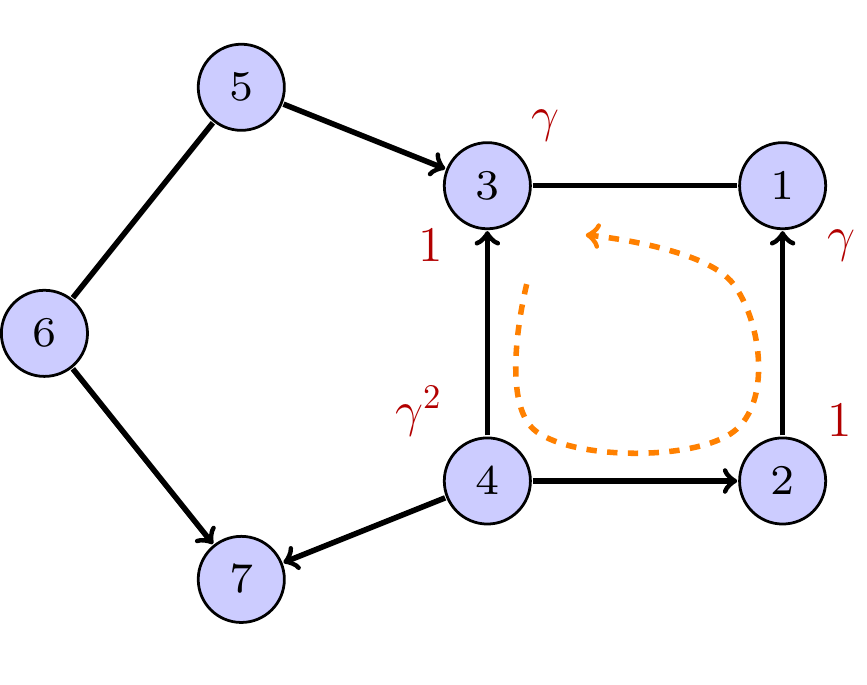}
\caption{Walk from/to vertex no.\ 3}
  \end{subfigure} \\
	\begin{subfigure}{.45\textwidth}
		\includegraphics[scale=0.6]{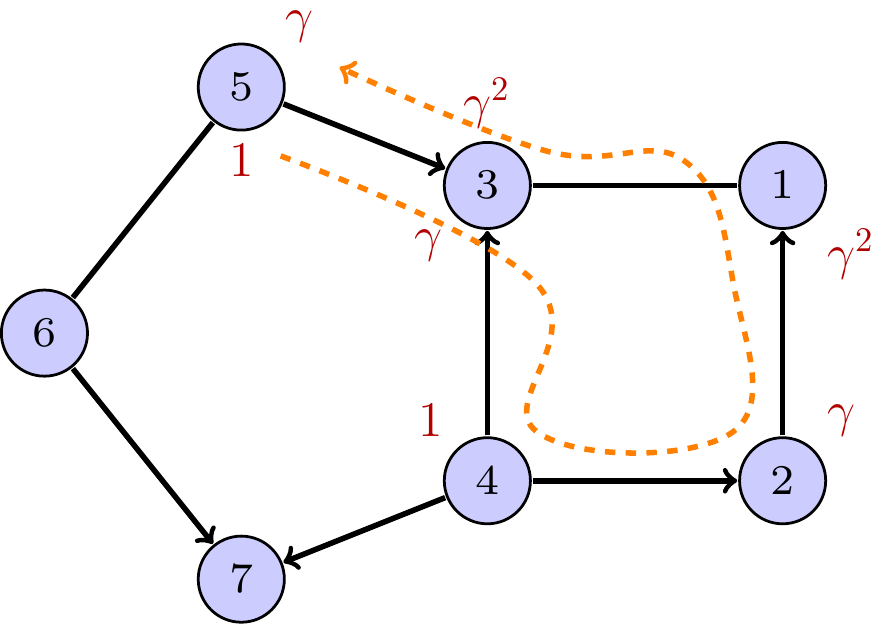}
\caption{Another walk from/to vertex no.\ 5}
	\end{subfigure}
	\caption{Values of $h_{\gamma}(W_j)$ for three closed mixed walks $W$}
	\label{fig:walk}
\end{figure}

Next, we state three useful basic properties of $h_{\gamma}$ with respect to mixed walks (some of which already implicit
in the previous section).

\begin{proposition}\label{fwbw}
Let $W$ be a mixed walk containing $r$ forward arcs and $s$ backward arcs. Then
$
 h_{\alpha}(W) = \alpha^r\bar\alpha^s.
$
\end{proposition}

\begin{proposition}\label{frev}
Let $W'$ be a mixed walk and $W''$ the corresponding reverse walk. Let $W$ be the walk
resulting from concatenating $W'$ and $W''$. Then
$
 h_{\alpha}(W) = 1.
$
\end{proposition}

\begin{proposition}\label{fconcat}
Let $W', W''$ be two mixed walks such that the final vertex of $W'$ is the start vertex of $W''$. Let $W$ be the walk
resulting from concatenating $W'$ and $W''$. Then
$
 h_{\alpha}(W) = h_{\alpha}(W') h_{\alpha}(W'').
$
\end{proposition}

As can be seen from \Cref{fig:walk}, $h_{\alpha}$ can be used to assign (possibly multiple) complex numbers to each of the
vertices of a mixed graph. In particular, we are concerned about the possible values that get assigned to the
start/end vertices of closed walks:

\begin{definition}\label{def:sammlung}
Let $D$ be a mixed graph.
The $\alpha$\hyp{}store $S^\alpha(u)$ of  $u\in V(D)$ is defined as
$
S^\alpha(u)=\{ h_{\alpha}(W):\ \text{$W$ is a closed walk in $D$ from/to $u$}\}.
$
Let $s^\alpha(u) = \vert S^\alpha(u) \vert$ denote the associated store size.
\end{definition}

Trivially, $1\in S^\alpha(u)$ so that $s^\alpha(u)\geq 1$. As the following \namecref{samesto} shows,
the store content is independent of $u$, hence we may speak of `the' $\alpha$\hyp{}store of a mixed graph:

\begin{lemma}\label{samesto}
Let $u,v\in V(D)$. Then
$
S^\alpha(u) = S^\alpha(v).
$
\end{lemma}

\begin{myproof}
Let $W''$ be a closed mixed walk from/to $v$. Clearly, $h_{\alpha}(W'')\in S^\alpha(v)$.
Now let $W'$ be a mixed walk from $u$ to $v$ and let $W'''$ its reverse walk.
Concatenating $W'$, $W''$, $W'''$ one obtains a closed mixed walk $W$ from/to $v$. 
Using \Cref{frev} and \Cref{fconcat} we get
$
  h_{\alpha}(W) = h_{\alpha}(W') h_{\alpha}(W'') h_{\alpha}(W''') = h_{\alpha}(W''),
$	
so that $h_{\alpha}(W'')\in S^\alpha(u)$. Repeating the argument with the roles of $u$ and $v$ swapped yields $S^\alpha(v)= S^\alpha(u)$.
\end{myproof}

\begin{theorem}\label{stchar}
Let $D$ be a connected mixed graph. Then the following statements are equivalent:
\begin{itemize}\item[]\begin{enumerate}[label=(\roman*)]
\item \label{stchar1} $s^\alpha(u)=1$ for at least one $u\in V(D)$.
\item \label{stchar2} $s^\alpha(u)=1$ for every $u\in V(D)$.
\item \label{stchar4} $h_{\alpha}(W')=h_{\alpha}(W'')$ for every pair $W',W''$ of mixed walks sharing the same start and end vertices.
\item \label{stchar5} $D$ is an $\alpha$\hyp{}monograph (of \kindone kind).
\end{enumerate}\end{itemize}
\end{theorem}

\begin{myproof}
Observing \Cref{def:mono1} and \Cref{def:sammlung}, this follows from \Cref{frev}, \Cref{fconcat} and \Cref{samesto}.
\end{myproof}

\begin{example}
The mixed graph shown in \Cref{fig:walk} is not a $\gamma$\hyp{}monograph, as opposed to the slightly different graph depicted
in \Cref{fig:walkmono}. Notice how, in view of \Cref{stchar}, every mixed walk with the same start/end vertex creates exactly the same walk values along the way.
\end{example}

\begin{figure}
	\centering
	\begin{subfigure}{.45\textwidth}
		\includegraphics[scale=0.6]{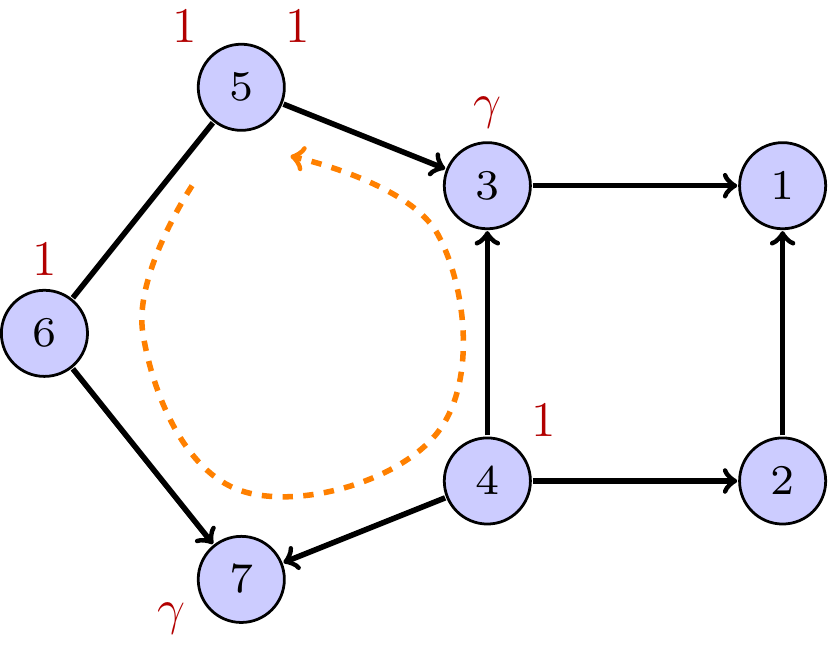}
\caption{Walk \#1 from/to vertex no.\ 5}
  \end{subfigure}
	\hspace*{0.5cm}
	\begin{subfigure}{.45\textwidth}
		\includegraphics[scale=0.6]{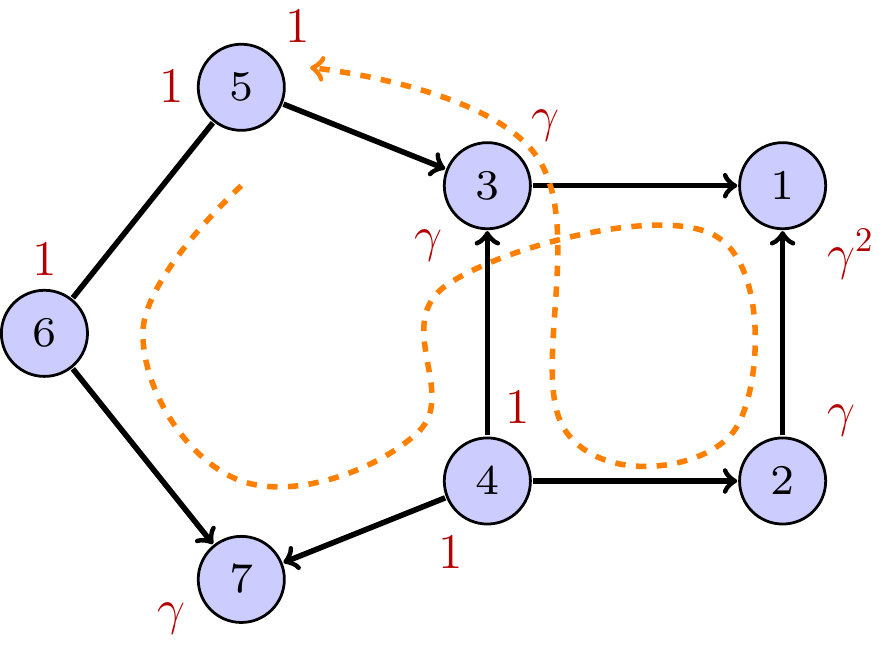}
\caption{Walk \#2 from/to vertex no.\ 5}\label{fig:wmono2}
  \end{subfigure} 
	\caption{Values of $h_{\gamma}(W_j)$ for two closed walks in a $\gamma$\hyp{}monograph}
	\label{fig:walkmono}
\end{figure}

In view of \Cref{stchar}, one can characterize $\alpha$\hyp{}monographs by the way their vertices can be partitioned:

\begin{theorem}\label{monopart} 
A connected mixed graph $D$ is an 
$\alpha$\hyp{}monograph (of \kindone kind) if and only if $V(D)$ can be partitioned
into sets $V_{\alpha^0},V_{\alpha^1},V_{\alpha^2},\ldots$ (some of which possibly empty) such that
there are no digons between any two sets $V_{\alpha^{j_1}}, V_{\alpha^{j_2}}$ with $j_1\not=j_2$ and
every arc starting in a set $V_{\alpha^j}$ ends in $V_{\alpha^{j-1}}$.
\end{theorem}

\begin{myproof}
Given an $\alpha$\hyp{}monograph $D$, use \Cref{stchar} \ref{stchar4} as follows. 
Fix any vertex $u\in V(D)$ and assign each vertex $v\in V(D)$ to the set $V_{\alpha^j}$, where $\alpha^j=h_\alpha(W)$ for
an arbitrary walk from $u$ to $v$. Conversely, consider a mixed graph with a vertex partition as supposed. 
We may assume $V_{0}\not=\emptyset$. Fix any $v\in V_{0}$. Considering some vertex 
$w\in V_{\alpha^{k}}$ and an arbitrary mixed walk $W$ from $v$ to $w$, we see that the partition structure aligns with
\eqref{eq:stfunc1} and \eqref{eq:stfunc2}, so that inductively we conclude $\alpha^k=h_\alpha(W)$. 
Thus, condition \ref{stchar4} of \Cref{stchar} is satisfied.
\end{myproof}

Furthermore, the store values of an $\alpha$\hyp{}monograph $D$ permit us to convert the eigenvectors of $\Gamma(D)$
into eigenvectors of $D$:

\begin{theorem}\label{monoeigtransfer}
Let $D$ be an $\alpha$\hyp{}monograph and $x=[x_u]_{u\in V(D)}$ an eigenvector of $\Gamma(D)$ for eigenvalue $\lambda$.
Fixing a reference vertex $v\in V(D)$, define the vector
	\begin{align}\label{eq:defy}
		y=[y_r]_{r\in V(D)}=\left[{h_\alpha(v\kzg r)}{x_r}\right]_{r\in V(D)},
  \end{align}
where  $v\kzg r$ is an arbitrary mixed walk from $v$ to $r$ in $D$ (cf.\ \Cref{stchar} \ref{stchar4}).
Then $y$ is an $\alpha$\hyp{}eigenvector of $D$ for eigenvalue $\lambda$.
\end{theorem}

\begin{myproof}
Clearly, for every vertex $u$, the vector $x$ satisfies the summation rule
	\begin{align}\label{eq:srux}
		\lambda x_u=\sum_{r\in N_{\Gamma(D)}(u)}x_r.
	\end{align}
	Using the recursion \eqref{eq:stfunc2} as well as equations \eqref{eq:defy} and \eqref{eq:srux}, we deduce:
	\begin{align}
	\lambda y_u &=	\lambda {h_\alpha(v\kzg u)}{x_u}\\ &={\sum_{r\in N_{\Gamma(D)}(u)}{h_\alpha(v\kzg u)}x_r}\\
		&={\sum_{r\in N_D(u)}{h_\alpha(v\kzg u)}x_r}+{\sum_{r\in N^+_{D}(u)}{h_\alpha(v\kzg u)}x_r}+{\sum_{r\in N^-_{D}(u)}{h_\alpha(v\kzg u)}x_r}\\
		&=\sum_{r\in N_D(u)}{h_\alpha(v\kzg r)}{x_r}+\alpha\sum_{r\in N^+_{D}(u)}{h_\alpha(v\kzg r)}{x_r}+\bar{\alpha}\sum_{r\in N^-_{D}(u)}{h_\alpha(v\kzg r)}{x_r}\\
		&=\sum_{r\in N_D(u)}y_r+\alpha\sum_{r\in N^+_{D}(u)}y_r+\bar{\alpha}\sum_{r\in N^-_{D}(u)}y_r.
	\end{align}
	Comparing this to the mixed summation rule \eqref{eq:srule} from \Cref{srule}, we see that $y$ is as claimed.
\end{myproof}

Regarding \Cref{monoeigtransfer}, note that the construction \eqref{eq:defy} retains linear independence, i.e.,
every basis of eigenvectors of $\Gamma(D)$ can be converted into a basis of $\alpha$\hyp{}eigenvectors of $D$ 
(using the same reference vertex $v$ throughout).

\begin{example}
\Cref{fig:wmono2} depicts an $\alpha$\hyp{}monograph $D$. With respect to the indicated vertex order,
$(-1,-1,1,1,0,-1,1)^T$ is an eigenvector of $\Gamma(D)$ for eigenvalue $0$. 
Using equation  \eqref{eq:defy} and the store values given in 
the figure, one obtains the $\gamma$\hyp{}eigenvector  $(-\gamma,-\gamma^2,\gamma^2,1,0,-1,\gamma^2)^T$
for eigenvalue $0$ of $D$.
\end{example}

In preparation for the following section we define a variant of the function $h_{\alpha}$. 
As before, let $D$ be a connected mixed graph. Fix $u\in V(D)$ and let $W$ be a mixed walk $u=r_1,\ldots,r_k$ in $D$. 
Slightly changing equations \eqref{eq:stfunc1} and \eqref{eq:stfunc2}, we define
\begin{align}
g_\alpha(W_j) & = (-1)^{j+1} h_{\alpha}({W_j}),
\end{align}
for $j=1,\ldots,k$.
The properties mentioned in \Cref{fwbw,frev,fconcat} also hold for $g_{\alpha}(W)$, thus
justifying an alternative notion of $\alpha$\hyp{}store. Using this notion and the following
\namecref{def:mono2} instead of \Cref{def:mono1}, one can check that \Cref{samesto} and \Cref{stchar} remain valid for $g_{\alpha}(W)$ as well.

\begin{definition}\label{def:mono2}
A mixed graph is an $\alpha$\hyp{}monograph (of \kindtwo kind) if $g_\alpha(\vec{C})=1$ for all its cycles $C$.
\end{definition}

As a result, we can derive results analogous to (but slightly different from)
\Cref{monopart} and \Cref{1stspecf}:

\begin{theorem}\label{monopart2} 
A connected mixed graph $D$ is an
$\alpha$\hyp{}monograph (of \kindtwo kind) if and only if $V(D)$ can be partitioned
into sets $\ldots,V_{-\alpha^2},V_{-\alpha^1},V_{-\alpha^0},V_{\alpha^0},V_{\alpha^1},V_{\alpha^2},\ldots$ (some of which possibly empty)
such that digons occur only between pairs of sets $V_{\alpha^j}$, $V_{-\alpha^j}$ and any arc
starting in a set $V_{\alpha^j}$ ends in $V_{\alpha^{j-1}}$.
\end{theorem}

\begin{theorem}\label{1stspecg} 
Let $D$ be an $\alpha$\hyp{}monograph (of \kindtwo kind). Then, $\sigma_\alpha(D)=-\sigma(\Gamma(D))$.
\end{theorem}

\begin{example}
The mixed graph $D$ shown in \Cref{fig:monostg} is an $i$\hyp{}monograph of \kindtwo kind.
We have $\sigma_i(D)=\{-3,1^{(3)}\}$ and $\sigma(\Gamma(D))=\{3,-1^{(3)}\}$.
Clearly, $D$ is not an $i$\hyp{}monograph of \kindone kind.
\end{example}

\begin{figure}
	\centering
		\includegraphics[scale=0.6]{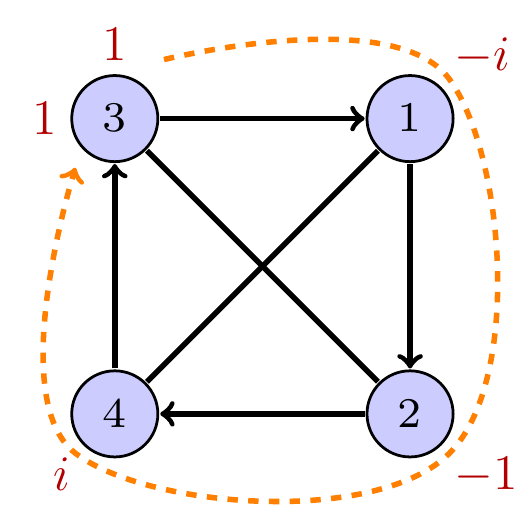}
\caption{An $i$\hyp{}monograph of \kindtwo kind}
	\label{fig:monostg}
\end{figure}

\section{\bf Spectral Radius}

The spectral radius $\rho(M)$ of a complex matrix $M$ is defined as the largest modulus among its eigenvalues.
If $\| \cdot \|$ is any matrix norm, we have $\rho(M) \leq \| M\|$ (cf.\ Theorem 5.6.9 in \cite{Horn}).
Using the maximum norm $\| \cdot \|_\infty$, one immediately obtains the classic upper bound 
$
\rho(G) \le \Delta_{G}
$
on the spectral radius $\rho(G):=\rho(A(G))$ of an undirected graph $G$. Supposing $G$ is connected, 
equality holds if and only if $G$ is regular. Considering a mixed graph $D$ and its $\alpha$\hyp{}Hermitian adjacency matrix $H_\alpha(D)$ instead,
it follows from \Cref{halpha} that 
$
\| H^\alpha(D) \|_\infty = \| A(\Gamma(D)) \|_\infty,
$
since all nonzero entries of $H^\alpha(D)$ have modulus $1$. Hence,
$
\rho_{\alpha}(D) \le \Delta_{\Gamma(D)}.
$
Interestingly, $\alpha$\hyp{}monographs come into play if one wants to characterize when equality holds:

\begin{theorem}\label{vinduc}
Let $D$ be a connected mixed graph. Then,
$
   \rho_{\alpha}(D) = \Delta_{\Gamma(D)}
$
if and only if $D$ is a regular $\alpha$\hyp{}monograph (of \kindone or \kindtwo kind).
\end{theorem}

\begin{myproof}
	Let $x=[x_u]_{u\in V(D)}$ be an $\alpha$\hyp{}eigenvector of $D$ for $\alpha$\hyp{}eigenvalue $\lambda$.
	Choose $v\in V(D)$ such that  $\vert x_v\vert$ is maximal. We may assume $\vert x_v\vert=1$.
	Using \Cref{srule}, we can deduce that
	\begin{align}
	\left\vert \lambda \right\vert = \left\vert \lambda x_{v} \right\vert  & 
	\le  \sum_{u\in N(v)} \vert x_{u}\vert  + \sum_{u\in N^+(v)} \vert \alpha x_{u}\vert  +  \sum_{u\in N^-(v)} \vert \overline{\alpha} x_{u}\vert  \label{A1}  \\[0.5em]
	& =  \sum_{u\in N(v)} \vert x_{u}\vert  + \sum_{u\in N^+(v)} \vert x_{u}\vert  +  \sum_{u\in N^-(v)} \vert x_{u}\vert   \\[0.5em]
	& \le  \sum_{u\in N(v)} \vert x_{v}\vert  + \sum_{u\in N^+(v)} \vert x_{v}\vert  +  \sum_{u\in N^-(v)} \vert x_{v}\vert  \label{B1}  \\[0.5em]
	& = \deg_{\Gamma(D)} (v)     \\[0.5em]
	& \le \Delta_{\Gamma(D)}.  \label{C1}
	\end{align}
	
	Suppose that $\lambda$ is an $\alpha$\hyp{}eigenvalue of $D$ with largest modulus.
	The condition $\rho_{\alpha}(D) = \Delta_{\Gamma(D)}$ holds if and only if equality holds in all three conditions \eqref{A1}, \eqref{B1} and \eqref{C1}.
	Equality in \eqref{C1} is achieved if any only if $D$ (resp.\ $\Gamma(D)$) is regular of degree $\Delta_{\Gamma(D)}$.
	Since $\vert x_{v}\vert$ is maximal, equality in \eqref{B1} occurs exactly if $\vert x_{u}\vert  = \vert x_{v}\vert = 1$ for all $u\in N_{\Gamma(D)}(v)$.
	Repeat this argument for all vertices $u\in N_{\Gamma(D)}$, each time taking the role of $v$. Since $\Gamma(D)$ is connected,
	we successively prove $\vert x_{u}\vert  = 1$ for all $u\in V(D)$.
	
	Equality holds in the complex triangle inequality \eqref{A1} if and only if  $\arg(\lambda x_{v}) = \arg(x_u)$ for all $u\in N_{\Gamma(D)}(v)$.
	In the following, we shall skip the trivial case $\lambda=0=\rho_{\alpha(D)}$.
	Let $\arg (\alpha) = \theta\in\mathbb{R}$. Consider the following cases:
	
	\begin{enumerate}[label=(\roman*)]
		\item \label{item11} Case $\lambda > 0$:
		\begin{itemize}
		\item If $u\in N(v)$, then $\arg (x_{u}) = \arg (\lambda x_{v})$, so that $x_{u} = x_{v}$.
		\item If $u\in N^+(v)$, then $\arg (\alpha x_{u}) = \arg (\lambda x_{v})$, so that $x_{u} = \bar\alpha x_{v}$.
		\item If $u\in N^-(v)$, then $\arg (\bar\alpha x_{u}) = \arg (\lambda x_{v})$, so that $x_{u} = \alpha x_{v}$.
		\end{itemize}
		Assigning each vertex $u\in V(D)$ to a set $V_\theta$ with $\theta=x_{u}/x_{v}$, we obtain a partition as mentioned in \Cref{monopart}.

		
		\item \label{item22} Case $\lambda < 0$:
		\begin{itemize}
		\item If $u\in N(v)$, then $\arg (x_{u}) = \arg (\lambda x_{v})$, so that $x_{u} = -x_{v}$.
		\item If $u\in N^+(v)$, then $\arg (\alpha x_{u}) = \arg (\lambda x_{v})$, so that $x_{u} = -\bar\alpha x_{v}$.
		\item If $u\in N^-(v)$, then $\arg (\bar\alpha x_{u}) = \arg (\lambda x_{v})$, so that $x_{u} = -\alpha x_{v}$.
		\end{itemize}
		Assigning each vertex $u\in V(D)$ to a set $V_\theta$ with $\theta=x_{u}/x_{v}$, we obtain a partition as mentioned in \Cref{monopart2}.


	\end{enumerate}
	
	Conversely, let $D$ be a connected mixed graph having a vertex partition according to one of the cases \ref{item11} or \ref{item22}.
	Construct a vector $x=[x_u]_{u\in V(D)}$ as follows. Set $x_u:= q$ for any $u\in V_q$. It is straightforward to show that $x$
	is an $\alpha$\hyp{}eigenvector for an eigenvalue of modulus $\rho_\alpha(D)$.
\end{myproof}


\begin{corollary}\label{partbip}
Let $D$ be a connected mixed graph. Suppose that $\alpha^k \not= -\alpha^l$ for all $k,l\in\mathbb{Z}$. 
Then,
$
   \rho_{\alpha}(D) = \Delta_{\Gamma(D)}
$
if and only if $D$ is a regular $\alpha$\hyp{}monograph of \kindone kind.
\end{corollary}

\begin{myproof}
The given condition on $\alpha$ guarantees that the partition arising in case \ref{item22} in the proof of \Cref{vinduc}
can be converted into a bipartition $V(D)=V' \dot{\cup} V''$, with $V'=V_{\alpha^0} \cup V_{\alpha^1} \cup \ldots$ and 
$V''=V_{-\alpha^0} \cup V_{-\alpha^1} \cup \ldots$, such that the subgraphs induced by $V'$ and $V''$ have no edges.
So $D$ would be bipartite in this case.
Forming pairwise unions $V_{\alpha^k}\cup V_{-\alpha^k}$, it becomes apparent that
$D$ must be an $\alpha$\hyp{}monograph of \kindone kind.
\end{myproof}

To conclude, let us briefly revisit the three special choices $i$, $\omega$ and $\gamma$ for $\alpha$.
Choosing $\alpha=\gamma$, the conditions of \Cref{partbip} are met, but not for $\alpha\in\{i, \omega\}$.
In view of this, $\gamma$ appears to be an interesting choice for $\alpha$. Moreover, since $\gamma$ is a third root of unity,
among all candidates satisfying the conditions of \Cref{partbip}, choosing $\alpha=\gamma$ will yields
a minimal number of sets in the monograph vertex partition.




\end{document}